\documentclass[12pt]{amsart}

\usepackage{latexsym, amssymb, amsmath,ulem,soul,esint}

\usepackage{amsfonts, graphicx}
\usepackage{graphicx,color}
\usepackage{url}
\newcommand{\be}{\begin{equation}}
\newcommand{\ee}{\end{equation}}
\newcommand{\beq}{\begin{eqnarray}}
\newcommand{\eeq}{\end{eqnarray}}

\newtheorem{thm}{Theorem}[section]

\newtheorem{lma}{Lemma}[section]
\newtheorem{prop}{Proposition}[section]

\theoremstyle{remark}
\newtheorem{rem}{Remark}[section]
\numberwithin{equation}{section}

\def\be{\begin{equation}}
\def\ee{\end{equation}}
\def\bee{\begin{equation*}}
\def\eee{\end{equation*}}

\def\lf{\left}
\def\ri{\right}

\newcommand{\Ric}{\mathrm{Ric}}

\def\Ric{\text{\rm Ric}}
\def\Rm{\text{\rm Rm}}

\def\p{\partial}

\def\heat{\lf(\frac{\p}{\p t}-\Delta_{g(t)}\ri)}

\def\e{\varepsilon}

\def\b{{\beta}}

\def\R{\mathbb{R}}

\begin{document}

\title{Gap Theorem on manifolds with small curvature concentration}

 \author[P.-Y. Chan]{Pak-Yeung Chan}
\address[Pak-Yeung Chan]{Mathematics Institute, Zeeman Building, University of Warwick, Coventry CV4 7AL, UK}
\email{pak-yeung.chan@warwick.ac.uk}

\author[M.-C. Lee]{Man-Chun Lee}
\address[Man-Chun Lee]{Department of Mathematics, The Chinese University of Hong Kong, Shatin, N.T., Hong Kong}
\email{mclee@math.cuhk.edu.hk}
%\thanks{M.-C. Lee is partially supported by direct grant}

\subjclass[2020]{Primary 53E20}

\date{\today}

\begin{abstract} 
In this work, we show that complete non-compact manifolds with non-negative Ricci curvature, Euclidean volume growth and sufficiently small curvature concentration are necessarily flat Euclidean space. 
\end{abstract}

\keywords{small curvature concentration}

\maketitle

\section{Introduction}
One of the fundamental goal of Riemannian geometry is to investigate the underlying structure of a manifold. In this work, we are interested in the structure of manifolds with pinched curvature in the term of curvature concentration. The consideration of manifolds with bounded integral curvature dates back to the study of volume non-collapsed Einstein manifolds or more generally manifolds with bounded Ricci curvature, for example see \cite{Anderson1989,AndersonCheeger1991,BandoKasueNakajima1989}. Curvature bound in form of $||\Rm||_{L^{n/2}}$ is particularly interesting since this is scaling invariant. In the content of Einstein manifolds, bounded curvature concentration often imposes sufficiently strong constraints so that only finitely many singularities can arise thanks to the Einstein structure, for example see \cite[Theorem 5.1]{BandoKasueNakajima1989}. 

In this work, we are interested in studying complete non-compact manifolds with $\Ric\geq 0$. In order to study its asymptotic, it is more often to consider its tangent cone at infinity. This is a pointed Gromov-Hausdorff limit of the sequence $(M^n,g_i,x_0)$ where $g_i=R_i^{-2}g$ for some divergent sequence $R_i\to +\infty$. The existence of such limit is guaranteed by Gromov's compactness Theorem. If $g$ is of Euclidean volume growth, the sequence of manifolds $(M,g_i,x_0)$ will be non-collapsed uniformly. If in addition $g$ is Ricci-flat with bounded $||\Rm||_{L^{n/2}}$, it follows from \cite[Theorem 1.5 \& Remark 2]{BandoKasueNakajima1989} that $M^n$ is isometric to $\mathbb{R}^n$ if $n$ is odd. This is in spirit similar to the gap Theorem of Anderson \cite{Anderson1990}. Motivated by this, we are interested to ask if a manifold is of $\Ric\geq 0$, Euclidean volume growth and has bounded curvature concentration, then is it necessarily flat? 
 In \cite[Theorem 6.10]{Cheeger2003},  a deep result of Cheeger together with Colding's volume continuity Theorem \cite{Colding1997} shows a gap Theorem under $\Ric\geq 0$, Euclidean volume growth and bounded curvature concentration when the dimension is odd. It is also immediate to observe from the Eguchi-Hanson metric that a general gap Theorem cannot be expected to hold for all dimension.   In comparison  with the Eguchi-Hanson metric, we consider manifolds with small curvature concentration which in theory rule out the singularities.  Indeed,  it was known by \cite[Theorem 8.3]{Cheeger2003} that this is the case under additional Ricci curvature decay assumption,  see also \cite[Theorem 4.32]{Cheeger2003}. It is therefore tempting to study the gap phenomenon {\sl without} any point-wise upper bound on the Ricci curvature but with small curvature concentration in \textbf{all dimension}.

Apart from the aforementioned static elliptic approach, the curvature concentration condition is shown to be a particularly well adapted quantity to the parabolic geometric flow. In \cite{ChanChenLee2021, ChanHuangLee2022}, a quantitative Ricci flow existence theory on manifold with small curvature concentration is developed and is applied  to investigate the topological rigidity of such manifolds without bounded curvature conditions (in fact under a weaker curvature conditions),  see also \cite{ChanLee2023} for a gap theorem without volume non-collapsing conditions.  In this work, we shall further exploit the monotonicity of the curvature concentration along the Ricci flow and show the following gap theorem:

\begin{thm} \label{thm:gap}
For any $n\geq 4,v>0$, there exists $\delta(n,v)>0$ such that if $(M^n,g_0)$ is a complete non-compact manifold with 
\begin{enumerate}
    \item[(i)] $\Ric(g_0)\geq 0$;
    \item[(ii)] The asymptotic volume ratio of $(M,g_0)$ satisfies $$\mathrm{AVR}(g) \geq v>0;$$
    \item[(iii)] Small curvature concentration, i.e. $$\left(\int_M |\Rm(g_0)|^\frac{n}{2}\, d\mathrm{vol}_{g_0} \right)^{\frac2n}<\delta,$$
\end{enumerate}
then $(M,g_0)$ is isometric to flat Euclidean space.
\end{thm}

The method of this work is based on the Ricci flow which is motivated by the following picture. By the volume comparison, it is known that the tangent cone at infinity is a metric cone thanks to the Euclidean volume growth. If the metric cone is smooth away from the tips, the curvature concentration must be unbounded unless it is flat. In particular, the Ricci flow coming out of the cone is not expected to have bounded curvature concentration unless it is flat. This will be made precise using the quantitative existence theory developed in \cite{ChanHuangLee2022}.

{\it Acknowledgement:} The first named author is supported by EPSRC grant EP/T019824/1. The second named author was partially supported by Hong Kong RGC grant (Early Career Scheme) of Hong Kong No. 24304222, a direct grant of CUHK and a NSFC grant.

\section{ Some preliminaries on Ricci flow}

We want to regularize $g_0$ using the Ricci flow and obtain contradiction from long-time asymptotic of the flow. We first collect some preliminaries of Ricci flow result.

\subsection{Perelman's entropy and its localization}
The entropy along the Ricci flow was introduced by Perelman \cite{Perelman2002} and its definition was further localized by Wang \cite{Wang2018}. Suppose $(M^n,g)$  is an $n$ dimensional complete (not necessarily compact) Riemannian manifold and $\Omega$ is a connected domain on $M$ with smooth boundary (boundaryless if $M=\Omega$). Similar to the notation in \cite{Wang2018}, we let
\begin{align*}
D_g(\Omega)&:=\left\{u: u\in W^{1,2}_0(\Omega), u\geq 0 \text{  and  } \|u\|_2=1 \right\},
\\
W(\Omega, g, u, \tau)&:=\int_{\Omega}\tau(\mathcal{R}\,u^2+4|\nabla u|^2)-2u^2\log u   \,  d\mathrm{vol}_g -\frac{n}{2}\log(4\pi\tau)-n,\notag
\\
\nu(\Omega, g, \tau)&:= \inf_{u\in D_g(\Omega), s \in (0, \tau]}W(\Omega, g, u, s),
\\
\nu(\Omega, g)&:= \inf_{\tau\in (0,\infty)}\nu(\Omega, g, \tau).
\end{align*}
We refer the readers to \cite{Perelman2002} and  \cite{Wang2018} for the (almost) monotonicity of these quantities along the bounded curvature Ricci flow and their applications. Analogously, we can define the standard entropy without the scalar curvature term, which is closely related to the classical Log Sobolev inequality.
\begin{align*}
\bar{W}(\Omega, g, u, \tau)&:=\int_{\Omega}4\tau |\nabla u|^2-2u^2\log u   \,  d\mathrm{vol}_g -\frac{n}{2}\log(4\pi\tau)-n,\notag
\\
\bar\nu(\Omega, g, \tau)&:= \inf_{u\in D_g(\Omega), s \in (0, \tau]}\bar{W}(\Omega, g, u, s),
\\
\bar\nu(\Omega, g)&:= \inf_{\tau\in (0,\infty)}\bar\nu(\Omega, g, \tau).
\end{align*}

\subsection{Existence of immortal Ricci flow from $g_0$}

We start with the long-time existence of the Ricci flow with no bounded curvature assumption obtained by the authors and Huang \cite{ChanHuangLee2022} (see also \cite{ChanChenLee2021} for the case of bounded curvature).
\begin{prop}\label{prop:longtime}
For any $n\geq 4,v,\e>0$, there exists $\delta_0(n,v,\e),C_0(n)>0$ such that if $(M^n,g_0)$ is a complete non-compact manifold with 
\begin{enumerate}
    \item[(i)] $\Ric(g_0)\geq 0$;
    \item[(ii)] The asymptotic volume ratio of $(M,g_0)$ satisfies $$\mathrm{AVR}(g)\geq v>0;$$
    \item[(iii)] $$\left(\int_M |\Rm(g_0)|^\frac{n}{2}\, d\mathrm{vol}_{g_0} \right)^{\frac2n}<\delta<\delta_0,$$
\end{enumerate}
then the asymptotic volume ratio of $g_0$ must satisfy $\mathrm{AVR}(g_0)\geq 1-\e$. Moreover, there exists a complete Ricci flow $g(t)$ on $M\times [0,+\infty)$ with $g(0)=g_0$ such that for all $t\in (0,+\infty)$,
\begin{enumerate}
\item[(a)] $\sup_M |\Rm(g(t))|\leq \e t^{-1}$;
\item[(b)] $\mathrm{Vol}_{g(t)}\left(B_{g(t)}(x_0,\sqrt{t} \right)\geq (1-\e)\omega_n t^{n/2}$ for all $x_0\in M$;
\item[(c)] The Perelman's $\nu$-entropy (of all scale) satisfies 
$$\nu(M,g(t))\geq -\e;$$
\item[(d)] $\displaystyle \left(\int_M |\Rm(g(t))|^\frac{n}{2}\, d\mathrm{vol}_{g(t)} \right)^{\frac2n}\leq C_0\delta$;
\item[(e)] For any compactly supported $u\in C^\infty_{c}(M,g(t))$,
$$\left(\int_M |u|^{\frac{2n}{n-2}}\, d\mathrm{vol}_{g(t)}\right)^{\frac{n-2}{n}}\leq C_0\left( \int_M4|\nabla u|_{g(t)}^2+\mathcal{R}_{g(t)} u^2\,  d\mathrm{vol}_{g(t)}\right);$$
\item[(f)] The scalar curvature satisfies $\mathcal{R}(g(t))\geq 0$, % for all $t>0$.
\end{enumerate}
where $\omega_n$ is the volume of the unit ball in $\R^n$. 
\end{prop}
\begin{proof}
This was implicitly proved in \cite[Theorem 1.2]{ChanHuangLee2022} and \cite{ChanChenLee2021}.  We give a sketch here. Since the Ricci curvature of $g_0$ is non-negative,  by \cite[Theorem 3.6]{Wang2018} the Euclidean volume growth $\mathrm{AVR}(g_0)\geq v$ implies that there exists $C(n,v)>0$ such that the (standard) $\bar \nu$-entropy of $g_0$ is bounded from below by $-C$ for all scale, i.e. $\bar\nu(M,g_0)>-C$.  It follows from \cite[Theorem 1.2]{ChanHuangLee2022} that the asymptotic volume ratio of $g_0$ satisfies $\mathrm{AVR}(M,g_0)\geq 1-\e$ if $\delta_0$ is sufficiently small and hence without loss of generality we might assume $\bar\nu(M,g_0)\geq -\e$ by \cite[Theorem 5.9]{Wang2020}.  By applying \cite[Theorem 1.1]{ChanHuangLee2022} to $R^{-2}g_0$ and rescaled the resulting solution back, we obtain a sequence of Ricci flow $g_R(t)$ with $g_R(0)=g_0$ on $M\times [0,T\cdot R^2]$ and desired estimate (a), (b), (d) and (f) for some uniform $T>0$ and for all $R\to+\infty$. The constant $C_0$ in (e) will depend only on $n$ since $\bar\nu(M,R^{-2}g_0)\geq -\e$ and condition (iii) is scaling invariant.  By \cite[Corollary 3.2]{Chen2009} and \cite[Theorem 14.16]{ChowBookII}, we might let $R\to+\infty$ to obtain the immortal solution $g(t)$ with (a), (b), (d) and (f), see the proof of \cite[Theorem 1.2]{ChanHuangLee2022}.  The Sobolev inequality (e) will follow from (c) using the argument in \cite{Ye2015,Zhang2007}. One can also apply a scaling argument using \cite[Theorem 2.2]{ChanChenLee2021} to show the uniform Sobolev inequality on $M$. If the flow has bounded curvature for all $t\geq 0$, (c) follows from the classical $\nu$-monotonicity of Perelman \cite{Perelman2002} using the heat kernel estimate from \cite{ChauTamYu2011}.
  To accommodate the  unboundedness of curvature as $t\to 0$, we might argue as follows.  Fix $\tau,t_0>0$, since $g(t)$ is of bounded curvature for $t>0$ and $g(t)\to g_0$ in $C^\infty_{loc}$ as $t\to 0$, \cite[Theorem 5.4]{Wang2018} implies
  \begin{equation}
  \begin{split}
   \nu \left(B_{g(t_0)}(x_0,8A\sqrt{t_0}),g(t_0),\tau \right)
&\geq -A^{-2}+\nu\left(B_{g_0}(x_0,20A\sqrt{t_0}),g_0,\tau+t_0 \right)\\
&\geq -A^{-2}+\nu\left(M,g_0,\tau+t_0 \right)\\
&\geq -A^{-2}-\e
  \end{split}
  \end{equation}
for all $A$ large. By letting $A\to +\infty$, since $\tau$ is arbitrary  we see that (c) holds, for instances see \cite[Proposition 2.3]{Wang2018} for the proof.
\end{proof}

\begin{rem}
It is important to note that the proof to our main result relies on existence of a long-time solution to the Ricci flow with small curvature concentration preserved. Such solutions has been constructed by the authors and Huang \cite{ChanHuangLee2022} when $n\geq 4$. When $n=3$, existence of Ricci flow from $g_0$ with $\Ric\geq 0$ and Euclidean volume growth was shown by Simon-Topping \cite{SimonTopping2021}. We expect that their constructed solution will remain to have small curvature concentration by carefully examining the partial Ricci flow construction. The three dimensional case can however be taken care using the deep result of Cheeger \cite{Cheeger2003} and thus we do not pursue the parabolic approach here. 
\end{rem}

\subsection{Positivity of scalar curvature }

Along the Ricci flow constructed from Proposition~\ref{prop:longtime}, it follows easily from the strong maximum principle that the scalar curvature of $g(t)$ is either strictly positive or $\Ric(g(t))\equiv 0$ for $t>0$. In the latter case, the Ricci flow is static and hence the curvature estimate in Proposition~\ref{prop:longtime} implies that $g(t)\equiv g_0$ is flat.  The following is from  \cite[Proposition 5.12]{Wang2020} which says that the scalar curvature of $g(t)$ is quantitatively positive if the initial metric $g_0$ is non-flat.  Since we will apply this to Ricci flow $g(t)$ with possibly unbounded curvature at $t=0$,  we need a slightly modified version.

\begin{prop}\label{prop:positive-R}
There exists $\e_0(n)>0$ such that the following holds. For any $\delta>0$, there exists $\frac12>\xi(n,\delta)>0$ and $S(n,\delta)>0$ such that if $g(t)$ is a complete Ricci flow on $M\times [0,r^2]$ with $g(0)=g_0$ for some $r>0$ and satisfies 
\begin{enumerate}
\item[(i)] $\nu(M,g(t),r^2)\geq -\e_0$ for all $t\in [0,r^2]$;
\item[(ii)] $\Ric(g_0)\geq -(n-1)\xi r^{-2}$ on $B_{g_0}(x_0,\xi^{-1}r)$;
\item[(iii)] $\mathrm{Vol}_{g_0}\left( B_{g_0}(x_0,s)\right)\leq (1-\delta)\omega_n s^n$ for all $s\in [r,\xi^{-1}r]$
\end{enumerate}
for some $x_0\in M$. Then the scalar curvature satisfies  $$\mathcal{R}(x,t)\geq \xi t^{-1}$$ for all $x\in B_{g(t)}(x_0,\sqrt{t})$ and $t\in [\frac12 r^2,r^2]$.  
\end{prop}

\begin{proof} 
The proof is almost identical to that of \cite[Proposition 5.12]{Wang2020}, we give a sketch for readers' convenience.  We may assume $r=1$ by considering $r^{-2}g(r^2t),t\in [0,1]$.  Suppose on the contrary, we can find a sequence of $\xi_i\to 0$ and $(M_i,g_i(t), x_i)$, $0\le t\le 1$ such that 
\begin{enumerate}
\item $\nu(M_i,g_i(t),1)\geq -\e_0$;
    \item $\Ric(g_i(0))\geq -(n-1)\xi_i\text{  on  } B_{g_i(0)}(x_i, \xi_i^{-1})$;
    \item $\mathrm{Vol}_{g_i(0)}\left(B_{g_i(0)}(x_i,s) \right)\leq (1-\delta)\omega_n s^n$ for all $s\in [1,\xi_i^{-1}]$;
 %   \item $|\Rm(g_i(t))|\leq \a t^{-1}$ on $M_i\times (0,1]$
\end{enumerate}
    and for some $t_i\in [1/2,1]$ and $y_i\in B_{g_i(t_i)}(x_i,\sqrt{t_i})$ we have
    \[
    t_i\mathcal{R}_i(y_i,t_i)< \xi_i\to 0.
    \]
    
    By \cite[Theorem 1.2]{Wang2020}, if $\e_0$ is sufficiently small, then each $g_i(t)$ satisfies 
    \begin{equation}
    |\Rm(g_i(t))|\leq t^{-1}\quad\text{and}\quad \mathrm{inj}(g_i(t))\geq \sqrt{t}
    \end{equation}
    on $M_i\times (0,1]$.  And hence we might apply Hamilton compactness theorem \cite{Hamilton1995} to see that after passing to subsequence,
    \[
    (M_i,g_i(t), x_i), t\in (0,1] \to (M_\infty,g_\infty(t), x_\infty), t\in (0,1]
    \]
    in the smooth Cheeger-Gromov sense.  We also have $t_i\to t_\infty \in [\frac12,1]$ and $y_i\to y_\infty\in \overline{B_{g_\infty(t_\infty)}(x_\infty,\sqrt{t_\infty})}$.     Moreover by the local maximum principle \cite[Theorem 1.1]{LeeTam2022} and Ricci lower bound at $t=0$, $R_{\infty}(x,t)\ge 0$ on $M_\infty\times (0,1]$.   So that $R_\infty(y_\infty,t_\infty)=0$.   By the strong maximum principle, $(M_\infty,g_\infty(t), x_\infty), 0< t\le 1$ is a static flow with $\Ric\equiv 0$. 
    Hence by Andersen gap theorem \cite{Anderson1990} (see also \cite[Theorem 3.5]{Wang2020}), $(M_{\infty}, g_{\infty}(1))$ is flat Euclidean space if $\e_0$ is small enough and thus the entire static flow $(M_{\infty}, g_{\infty}(t))$ is isometric to $\R^n$ for $t>0$.  In particular, this shows that 
    \begin{equation}
    \lim_{i\to+\infty}\int^1_0 \int_{B_{g_i(0)}(x_i,L)} |\mathcal{R}_i|\, d\mathrm{vol}_{g_i(\tau)} d\tau=0
    \end{equation}
    for all $L>0$ by considering the  evolution of volume, see \cite[(5.28)]{Wang2020}. This will imply 
    \[
    (M,g_i(0),x_i)\to (\R^n, \delta, z_\infty)
    \]
    in pointed Gromov Hausdorff sense, where $\delta$ denotes the flat metric.  But this contradicts with (3) by Colding's Volume Convergence Theorem \cite{Colding1997}.  This completes the proof of the proposition.
%   
%To show the estimate for $x\in B_{g(t)}(x_0,\sqrt{t})$ for $t\in [\frac12r^2,r^2]$ where we will assume $r=1$ by parabolic rescaling. We first note that we have $|\Rm(g(t))|\leq t^{-1}$ if $\e_0$ is sufficiently small by \cite[Theorem 1.2]{Wang2020} and hence  by distance distortion \cite[Lemma 8.3]{Perelman2002},  $d_{g_0}(x,x_0)\leq d_{g(t)}(x,x_0)+\b\sqrt{t}$ for some dimensional constant $\b(n)>1$.  In particular for all $s\in [1,\xi^{-1}]$,
%\begin{equation}
%\begin{split}
%\frac{\mathrm{Vol}_{g_0}\left( B_{g_0}(x,s)\right)}{s^n}&\leq \frac{\mathrm{Vol}_{g_0}\left( B_{g_0}(x_0,s+2\b\sqrt{t})\right)}{s^n}\\
%&\leq (1-\delta) \omega_n \cdot \left(1+ 2\b\sqrt{t}\right)^n\\
%&\leq \left(1-\frac12 \delta\right)\omega_n
%\end{split}
%\end{equation}
%provided that $t\leq S$ for some small $S(n,\delta)>0$.  
\end{proof}

\subsection{Monotonicity of $||\Rm||_{n/2}$ along Ricci flows}

Next, we need a global monotonicity of curvature concentration. The following 
localized form is proven in \cite[Lemma 2.1]{ChanChenLee2021}.
\begin{lma}\label{lma:mono-local}
Suppose $n\geq 3$ and $(M,g(t))$ is a complete solution to the Ricci flow for $t\in [0, T]$. Then there exist $C_0(n)>0$ such that for any compactly supported function $\phi(x,t)$ in space-time, we have
\begin{equation}
    \begin{split}
        \frac{d}{dt}\int_M \phi^2|\Rm|^{\frac{n}{2}}\; d\mathrm{vol}_{g(t)} \leq
  &      -C_0^{-1}\int_M|\nabla (\phi |\Rm|^{\frac{n}{4}})|^2\; d\mathrm{vol}_{g(t)}\\
& +C_0\int_M \phi^2|\Rm|^{\frac{n}{2}+1}\; d\mathrm{vol}_{g(t)}\\
& +C_0\int_M  |\nabla \phi|^2|\Rm|^{\frac{n}{2}}\; d\mathrm{vol}_{g(t)}\\
&+\int_M 2\phi |\Rm|^{\frac{n}{2}} \heat \phi  \cdot 
d\mathrm{vol}_{g(t)}.
    \end{split}
\end{equation}
\end{lma}
\vskip0.3cm

%
%Let $(M^n,g)$ be a complete Riemannian manifold satisfying the following conditions:
%\begin{enumerate}
%    \item $\Ric\ge 0$
%    \item $\nu(M,g)\ge -A$
%    \item $\int_M |\Rm|^\frac{n}{2}\,d\mu_g\le \delta$.
%\end{enumerate}
%It follows from the Ricci flow existence result and the local entropy monotoncity that if $0<\delta\le \delta_0(n,A)$, then there exist $C_0=C_0(n,A)>0$ and ;ong time solution to the Ricci flow $g(t)$ on $M\times [0,\infty)$ with $g(0)=g$ such that for all $t\ge 0$, $g(t)$ has nonnegative scalar curvature,
%\begin{eqnarray*}
%    |\Rm|&\le& t^{-1}\\
%  E(t):=\int_M |\Rm|^\frac{n}{2}\,d\mu_t&\le& C_1\delta\\
%    \nu(M, g(t))&\ge& -A\\
%\text{Vol}_t(B_t(x,\sqrt{t}))&\ge & c_0 t^{\frac{n}{2}}???
%\end{eqnarray*}
%We shall consider the situation that $g(t)$ is not the static Ricci flat flow and show that the Energy $\int_M |\Rm|^\frac{n}{2}\,d\mu_t$ satisfies a nice monotoncity formula. The $\nu$ entropy lower bound at all scales also implies a uniform Sobolev inequality along the flow: for all $u\in W^{1,2}(M)$
%\[
%\left(\int_M |u|^{\frac{2n}{n-2}}\,d\mu_t\right)^{\frac{n-2}{n}}\leq ce^{\frac{2A}{n}}\int_M4|\nabla u|^2+R u^2\, d\mu_t.
%\]
%for some positive dimensional constant $c$.

We now apply Lemma~\ref{lma:mono-local} to the Ricci flow $g(t)$ obtained from Proposition~\ref{prop:longtime} to obtain a monotonicity of global $||\Rm(g(t))||_{n/2}$. We denote  
\begin{equation}
E(t):=\int_M |\Rm(g(t))|^{n/2} \,d\mathrm{vol}_{g(t)}  
\end{equation}
for $t\in [0,+\infty)$.

\begin{lma}\label{lma:mono}
Under the assumption of Proposition~\ref{prop:longtime}, we have 
\begin{equation}
E(t)+\int_s^t \int_M |\Rm|^{\frac{n}{2}+1} \,d\mathrm{vol}_{g(\tau)}\,d\tau\leq  E(s)
\end{equation}
for all $0\leq s\leq t<+\infty$.
\end{lma}
\begin{proof}
We fix $0\leq s<t<+\infty$ and $x_0\in M$. In what follows, we will use $C_i$ to denote constants which depend only on $n$. Let $\eta(x,t)=d_{g(t)}(x,x_0)+c_n\sqrt{t}$ and define $\phi(x,t)=e^{-10t\rho^{-2}}\varphi\left(\frac{\eta(x,t)}{\rho}\right)$ where $\rho\gg 1$  and $\varphi(s)$ is a cutoff function on $\mathbb{R}$ so that $\varphi\equiv 1$ on $(-\infty,\frac12]$, $\varphi\equiv 0$ outside $(-\infty,1]$ and satisfies $\varphi''\geq -10\varphi $, $0\geq \varphi'\geq -10\sqrt{\varphi}$. By choosing $c_n$ large enough, we have from \cite[Lemma 8.3]{Perelman2002} that  
\begin{eqnarray}
\heat \phi \leq 0.
\end{eqnarray}
Moreover,  $
    |\nabla_{g(t)} \phi|^2\le C_n \rho^{-2}$
and $1\ge \phi\to 1$ on $M\times[s,t]$ as $\rho\to+\infty$.

Thanks to Lemma \ref{lma:mono-local}, 
\begin{equation}
\begin{split}
       \frac{d}{dt}\int_M \phi^2|\Rm|^{\frac{n}{2}}\,d\mathrm{vol}_{g(t)}  \leq& C_1 \delta^{n/2} \rho^{-2}
   -C_1^{-1}\int_M|\nabla (\phi |\Rm|^{\frac{n}{4}})|^2\,d\mathrm{vol}_{g(t)}  \\
& +C_1\int_M \phi^2|\Rm|^{\frac{n}{2}+1}\,d\mathrm{vol}_{g(t)}\\
=&C_1 \delta^{n/2} \rho^{-2}
   -C_1^{-1}\int_M|\nabla (\phi |\Rm|^{\frac{n}{4}})|^2\,d\mathrm{vol}_{g(t)}  \\
& +(C_1+1)\int_M \phi^2|\Rm|^{\frac{n}{2}+1}\,d\mathrm{vol}_{g(t)}-\int_M \phi^2 |\Rm|^{\frac{n}{2}+1}\,d\mathrm{vol}_{g(t)}.
\end{split}
\end{equation}

By H\"older inequality,
\begin{equation}\label{eqn:R-high-order}
\begin{split}
&\quad  (C_1+1)\int_M \phi^2|\Rm|^{\frac{n}{2}+1}\,d\mathrm{vol}_{g(t)} \\
 &\leq (C_1+1) \left(\int_M |\Rm|^{\frac{n}{2}}\,d\mathrm{vol}_{g(t)}  \right)^{\frac{2}{n}}\left(\int_M \phi^{\frac{2n}{n-2}}|\Rm|^{\frac{n^2}{2(n-2)}}\,d\mathrm{vol}_{g(t)}  \right)^{\frac{n-2}{n}}   \\
 &\leq C_2 \delta \left(\int_M \phi^{\frac{2n}{n-2}}|\Rm|^{\frac{n^2}{2(n-2)}}\,d\mathrm{vol}_{g(t)}\right)^{\frac{n-2}{n}}.
\end{split}
\end{equation}

On the other hand, we use the uniform Sobolev inequality, i.e. (e) %(iv) 
in Proposition~\ref{prop:longtime} together with \eqref{eqn:R-high-order}, to see that if we further require $\delta$ to be smaller than a dimensional constant $\delta_0(n)$, then
\begin{equation}
\begin{split}
&\quad \int_M|\nabla (\phi |\Rm|^{\frac{n}{4}})|^2\,d\mathrm{vol}_{g(t)} \\
&\ge C_3^{-1}\left(\int_M \phi^{\frac{2n}{n-2}}|\Rm|^{\frac{n^2}{2(n-2)}}\,d\mathrm{vol}_{g(t)}\right)^{\frac{n-2}{n}}-\frac{1}{4}\int_M \phi^2 \mathcal{R} |\Rm|^{\frac{n}{2}}\,d\mathrm{vol}_{g(t)}\\
&\ge  \left(C_3^{-1}-C_4\delta\right)\left(\int_M \phi^{\frac{2n}{n-2}}|\Rm|^{\frac{n^2}{2(n-2)}}\,d\mathrm{vol}_{g(t)}\right)^{\frac{n-2}{n}}\\
&\ge\frac{1}{2C_3}\left(\int_M \phi^{\frac{2n}{n-2}}|\Rm|^{\frac{n^2}{2(n-2)}}\,d\mathrm{vol}_{g(t)}\right)^{\frac{n-2}{n}}.
\end{split}
\end{equation}

Therefore, we arrive at the following almost monotonicity inequality: if we shrink $\delta_0$ further, then 
\begin{equation}
\begin{split}
\frac{d}{dt}\int_M \phi^2|\Rm|^{\frac{n}{2}}\,d\mathrm{vol}_{g(t)} &\leq  C_1\delta^{n/2}\rho^{-2} -\frac{1}{2C_3C_1}\left(\int_M \phi^{\frac{2n}{n-2}}|\Rm|^{\frac{n^2}{2(n-2)}}\,d\mathrm{vol}_{g(t)}\right)^{\frac{n-2}{n}}\\
& +C_2\delta \left(\int_M \phi^{\frac{2n}{n-2}}|\Rm|^{\frac{n^2}{2(n-2)}}\,d\mathrm{vol}_{g(t)}\right)^{\frac{n-2}{n}}\\
&- \int_M \phi^2 |\Rm|^{\frac{n}{2}+1}\,d\mathrm{vol}_{g(t)} \\
&\leq -\int_M \phi^2 |\Rm|^{\frac{n}{2}+1}\,d\mathrm{vol}_{g(t)} +C_1\rho^{-2}.
\end{split}
\end{equation}

By  integrating the above differential inequality over $[s,t]$, we see that 
\begin{equation}
\begin{split}
&\quad \int_M \phi^2|\Rm|^{\frac{n}{2}}\,d\mathrm{vol}_{g(t)}+\int_s^t\int_M \phi^2 |\Rm|^{\frac{n}{2}+1}\,d\mathrm{vol}_{g(\tau)}\,d\tau\\
&\leq \int_M \phi^2|\Rm|^{\frac{n}{2}}\,d\mathrm{vol}_{g(s)}+C_1\rho^{-2}(t-s).
\end{split}
\end{equation}

Since $|\Rm(g(t))|\in L^{n/2}$ for all $t\geq 0$, result follows from the dominated convergence theorem and by letting $\rho\to +\infty$.
\end{proof}

With the above preparations, we are now ready to prove Theorem~\ref{thm:gap}. 
\begin{proof}[Proof of Theorem \ref{thm:gap}]

We let $g(t)$ be the immortal Ricci flow constructed from Proposition~\ref{prop:longtime} with initial data $g(0)=g_0$. We will assume $\delta$ to be sufficiently small. We first note from Proposition~\ref{prop:longtime}, for any $\e>0$, we might choose $\delta$ so that 
\begin{equation}
\mathrm{AVR}(M,g_0)=\lim_{r\to+\infty}\frac{\mathrm{Vol}_{g_0}\left(B_{g_0}(x_0,r) \right)}{\omega_n r^n}\geq 1-\e.
\end{equation}

Suppose on the contrary that $g_0$ is non-flat. From $\Ric\geq 0$ and the rigidity of volume comparison Theorem we know that the asymptotic volume ratio of $g_0$ satisfies
\begin{equation}
\mathrm{AVR}(M,g_0)=1-\b
\end{equation}
for some $\b>0$ which will be small by shrinking $\delta$. Fix $x_0\in M$, there exists $r_0>0$ such that for all $r>r_0$,
\begin{equation}
1-\frac\b2 \geq \frac{\mathrm{Vol}_{g_0}\left(B_{g_0}(x_0,r) \right)}{\omega_n r^n}\geq 1-\b.
\end{equation}

By Proposition~\ref{prop:positive-R}, there exists $\xi>0$ so that for all $x\in B_{g(t)}(x_0,\sqrt{t})$ and $t>r_0^2$, 
\begin{equation}
\mathcal{R}(x,t)\geq \xi t^{-1}.
\end{equation}

Using also 
\begin{equation}
\mathrm{Vol}_{g(t)}\left( B_{g(t)}(x_0,\sqrt{t}) \right)\geq (1-\e)\omega_n t^{n/2}
\end{equation}
for all $t>0$, 
we see that for all $t>r_0^2$,
\begin{equation}
\begin{split}
& \quad \int_{r_0^2}^t \int_M |\Rm|^{\frac{n}{2}+1} \,d\mathrm{vol}_{g(\tau)}\,d\tau\\
&\geq  \int_{r_0^2}^t\int_{B_{g(\tau)}(x_0,\sqrt{\tau})}|\Rm|^{\frac{n}{2}+1} \,d\mathrm{vol}_{g(\tau)}\,d\tau\\
&\geq (1-\e) \omega_n \xi^{\frac{n}2+1} \int_{r_0^2}^t \tau^{-1} d\tau=  (1-\e) \omega_n \xi^{\frac{n}2+1} \log \left(\frac{t}{r_0^2} \right).
\end{split}
\end{equation}

But this contradicts with Lemma~\ref{lma:mono} since the right hand side tends to $+\infty$ as $t\to+\infty$ while the left hand side is uniformly bounded. This completes the proof of Theorem \ref{thm:gap}.
\end{proof}


\begin{thebibliography}{10}
%\bibitem{Brown1961} Brown, M.,{\sl The monotone union of open $n$-cells is an open n-cell}, Proc. Amer. Math. Soc. 12(1961), 812–814.
%

\bibitem{Anderson1989}Anderson, M. T. {\sl Ricci curvature bounds and Einstein metrics on compact manifolds}. J. Amer. Math. Soc. 2 (1989), no. 3, 455--490. 


\bibitem{Anderson1990}Anderson, M. T., {\sl Convergence and rigidity of manifolds under Ricci curvature bounds}. Invent. Math. 102 (1990), no. 2, 429--445. 

\bibitem{AndersonCheeger1991}Anderson, M. T.; Cheeger, J., {\sl Diffeomorphism finiteness for manifolds with Ricci curvature and $L^{n/2}$-norm of curvature bounded}. Geom. Funct. Anal. 1 (1991), no. 3, 231--252. 

\bibitem{BandoKasueNakajima1989}Bando, S.; Kasue, A.; Nakajima, H.,  {\sl On a construction of coordinates at infinity on manifolds with fast curvature decay and maximal volume growth}. Invent. Math. 97 (1989), no. 2, 313--349.


\bibitem{ChanChenLee2021}Chan, P.-Y.; Chen, E.; Lee, M.-C. , {\sl Small curvature concentration and Ricci flow smoothing}. J. Funct. Anal. 282 (2022), no. 10, Paper No. 109420, 29 pp.


\bibitem{ChanHuangLee2022}Chan, P.-Y.; Huang, S.-C.; Lee, M.-C., {\sl Manifolds with small curvature concentration}, arXiv:2207.07495.


\bibitem{ChanLee2023} Chan, P.-Y.; Lee, M.-C., {\sl Gap Theorem on Riemannian manifolds using Ricci flow}, arXiv:2305.01396

\bibitem{ChauTamYu2011}Chau, A.; Tam, L.-F.; Yu, C., {\sl  Pseudolocality for the Ricci flow and applications}, Canad. J. Math. 63 (2011), no. 1, 55-85.




%
%\bibitem{Chen2019}Chen, E., {\sl Convergence of the Ricci flow on asymptotically flat manifolds with integral curvature pinching}, arXiv:1907.13189. To appear in Ann. Sc. Norm. Super. Pisa Cl. Sci.
%
%
%\bibitem{CheegerColding1997}Cheeger, J.; Colding, T. H., {\sl On the structure of spaces with Ricci curvature bounded below. I.}, J. Diff. Geom. 46 (1997), no. 3, 406–480.
%
\bibitem{Cheeger2003}Cheeger, J., {\sl Integral bounds on curvature elliptic estimates and rectifiability of singular sets}, Geom. Funct. Anal. 13 (2003), no. 1, 20–72.
%
%\bibitem{ChenZhu2005}Chen, B.-L.; Zhu, X.-P.,  {\sl Volume growth and curvature decay of positively curved K\"ahler manifolds}. Q. J. Pure Appl. Math. 1 (2005), no. 1, 68--108.
%
%
\bibitem{Chen2009} Chen, B.-L., {\sl Strong uniqueness of the Ricci flow}, J. Differential Geom. \textbf{82} (2009), no. 2, 363--382, MR2520796, Zbl 1177.53036.
%
%\bibitem{Colding1997}Colding, T. H., {\sl Ricci curvature and volume convergence}. Ann. of Math. (2) 145 (1997), no. 3, 477--501. 
%
%
%
%\bibitem{CheegerGromov1985} Cheeger, J.; Gromov, M., {\sl Bounds on the von Neumann dimension of L2-cohomology and the Gauss-Bonnet theorem for open manifolds}, J. Differential Geom. 21 (1985), no. 1, 1–-34. 
%
%
%
%
\bibitem{ChowBookII}  Chow, B; Chu, S.-C.; Glickenstein, D.; Guenther, C.; Isenberg, J.; Ivey, T.; Knopf, D.; Lu, P.; Luo, F.; Ni, L., {\sl  Ricci flow: Techniques and Applications: Part II: Analytic aspects}. ‘Mathematical Surveys and Monographs,’ \textbf{144} A.M.S. 2008.


\bibitem{Colding1997}Colding,  T.H., {\sl Ricci curvature and volume convergence}, Ann. math. 145(1997), 477-501.


%
%
%\bibitem{Davies1989} Davies, E. B., {\sl Heat Kernel and Spectral Theory}. Cambridge University Press, 1989.
%
%%\bibitem{MunteanuWang2022} Munteanu, O.; Wang, J., {\sl Geometry of three-dimensional manifolds with scalar curvature lower bound} arXiv:2201.05595
%
%%\bibitem{NiTam2003} Ni, L.; Tam, L.-F., {\sl Plurisubharmonic functions and structure of complete K\"ahler manifolds with nonnegative curvature}, J. Differential Geom., 64 (2003), 457--524.
%
%\bibitem{GrigorSurvey}Grigor’yan, A., {\sl Estimates of heat kernels on Riemannian manifolds}, in ”Spectral Theory and Geometry. ICMS Instructional Conference, Edinburgh, 1998”, ed. B. Davies and Yu. Safarov, Cambridge Univ. Press, London Math. Soc. Lecture Notes 273 (1999) 140-225.
%
%\bibitem{GGP2017}Guidetti, D.; G\"uneysu; Pallara, D., {\sl $L^1$-elliptic regularity and $H = W$ on the whole $L^p$-scale on arbitrary manifolds}, Ann. Acad. Sci. Fenn., Math. 42 (2017) 497–521, 
%
%
%
%\bibitem{Hochard2016}Hochard, R., {\sl Short-time existence of the Ricci flow on complete, non-collapsed 3- manifolds with Ricci curvature bounded from below}, preprint, arXiv:1603.08726.
%

\bibitem{Hamilton1995}Hamilton, R.,  {\sl A compactness property for solutions of the Ricci flow}, Amer. J. Math. 117 (1995), no. 3, 545-572.



%\bibitem{He2016} He, F. , {\sl Existence and applications of Ricci flows via pseudolocality}, arXiv: 1610.01735.
%
%\bibitem{HondaMariRimoldiVeronelli}Honda, S., Mari, L., Rimoldi, M., Veronelli, G., {\sl  Density and non-density of $C^\infty_c\hookrightarrow W^{k,p}$, on complete manifolds with curvature bounds}. Nonlinear Analysis, 211, 112429.
%
%
%\bibitem{Huang2019} Huang, S., {A note on existence of exhaustion functions and its applications}, J. Geometric Analysis, (2019) 29:1649--1659.
%
%\bibitem{HuangKongRongXu2020} Huang, H.-Z.; Kong, L.-L.; Rong, X.-C.; Xu, S.-C., {\sl Collapsed manifolds with Ricci bounded covering geometry}. Trans. Amer. Math. Soc. 373 (2020), no. 11, 8039–8057.
%
%
%\bibitem{ImperaRimoldiVeronelli2022} Impera, D.; Rimoldi, M.; Veronelli, G., {\sl Higher order distance-like functions and Sobolev spaces}. Adv. Math. 396 (2022), Paper No. 108166, 59 pp.
%
%
%\bibitem{Ledoux1999}Ledoux, M., {\sl On manifolds with non-negative Ricci curvature and Sobolev inequalities}. Comm. Anal. Geom. 7 (1999), no. 2, 347–-353.
%
%\bibitem{LeeTam2017} Lee, M.-C.; Tam, L.-F., {\sl On existence and curvature estimates of Ricci flow}, arXiv:1702.02667.
%
%\bibitem{LeeTam2020} Lee, M.-C.; Tam, L.-F., {\sl Chern-Ricci flows on noncompact manifolds}, J. Differential Geom. 115 (2020), no. 3, 529--564.
%
\bibitem{LeeTam2022}Lee, M.-C.; Tam, L.-F., {\sl Some local maximum principles along Ricci flows}. Canad. J. Math. 74 (2022), no. 2, 329--348.
%
%
%\bibitem{Li2012} Li, P., {\sl  Geometric analysis}. Cambridge Studies in Advanced Mathematics, 134. Cambridge University Press, Cambridge, 2012. 
%
%
%
%\bibitem{Li-2012} Li, Y., {\sl Smoothing Riemannian metrics with bounded Ricci curvatures in dimension four}, II. Ann. Global Anal. Geom. 41 (2012), no. 4, 407-421.
%
%
\bibitem{Perelman2002}Perelman, G., {\sl The entropy formula for the Ricci flow and its geometric applications}, arXiv:math.DG/0211159.
%
%
%\bibitem{PetersonWei1997} Petersen, P.; Wei,G., {\sl  Relative volume comparison with integral curvature bounds}, Geom. Funct. Anal. 7(1997), 1031--1045.
%
%\bibitem{PetersonWei2001} Petersen, P.; Wei,G., {\sl Analysis and Geometry on Manifolds with Integral Ricci Curvature Bounds. II}, Trans. Amer. Math. Soc. 353(2001), no.2, 457--478.
%
%
%
%\bibitem{SchoenYau1994}Schoen, R.;Yau, S.-T., {\sl Lectures on Differential Geometry}, International Press, 1994
%
%\bibitem{ShiPhd}Shi, W.-X., {\sl Ricci deformation of the metric on complete noncompact K\"ahler manifolds}, Ph.D. thesis, Harvard University, 1990.
%
%\bibitem{Shi1989} Shi, W.-X., {\sl Deforming the metric on complete Riemannian manifold}, J. Differential Geom. 30 (1989), no. 1, 223--301.
%
%
\bibitem{SimonTopping2021} Simon, M.; Topping, P. M., {\sl Local mollification of Riemannian metrics using Ricci flow, and Ricci limit spaces}. Geom. Topol. 25 (2021), no. 2, 913--948. 


%
%
%\bibitem{Simon2008}Simon, M., {\sl Local results for flows whose speed or height is bounded by $c/t$}, Int. Math. Res. Not. IMRN 2008, Art. ID rnn 097, 14 pp, MR2439551, Zbl 1163.53042.
%
%
%\bibitem{Stallings1962}Stallings, J., {\sl The piecewise-linear structure of Euclidean space}. Proc. Cambridge Philos. Soc. 58 (1962), 481--488.
%
%
%\bibitem{Tam2010} Tam, L.-F., {\sl Exhaustion functions on complete manifolds}, Recent advances in geometric analysis, 211--215, Adv. Lect. Math. (ALM), 11, Int. Press, Somerville, MA, 2010.
%
%
%\bibitem{Topping2010}Topping, P. M., {\sl Ricci flow compactness via pseudolocality, and flows with incomplete initial metrics}, J. Eur. Math. Soc. (JEMS) 12 (2010), no. 6, 1429--1451.
%
%
%\bibitem{Wang2011}Wang, Y., {\sl Pseudolocality of the Ricci flow under integral bound of curvature}, J. Geom. Anal. {\bf 23} (2013), no. 1, 1--23. 
%
%
%
\bibitem{Wang2018}Wang, B., {\sl The local entropy along Ricci flow Part A: the no-local-collapsing theorems}. Camb. J. Math. {\bf 6} (2018), no. 3, 267-346.
%
\bibitem{Wang2020}Wang, B., {\sl The local entropy along Ricci flow Part B: the pseudo-locality theorems}. arXiv:2010.09981.
%
%
%\bibitem{Xia1}Xia, C., {\sl Open manifolds with nonnegative Ricci curvature and large volume growth}, Comment. Math. Helv. 74 (1999), 456–466.
%
%
%
%\bibitem{Xia2}Xia, C., {\sl Complete manifolds with nonnegative Ricci curvature and almost best Sobolev constant}, Illinois J. Math. 45 (2001), no. 4, 1253–1259.
%
%
%\bibitem{Xu2013}Xu, G., {\sl Short-time existence of the Ricci flow on noncompact Riemannian manifolds}. Trans. Amer. Math. Soc. 365 (2013), no. 11, 5605--5654.
%
%\bibitem{Yang2011}Yang, Y.-Y., {\sl Smoothing metrics on closed Riemannian manifolds through the Ricci flow}. Ann. Global Anal. Geom. 40 (2011), no. 4, 411--425.
%
%
\bibitem{Ye2015}Ye, R.-G., {\sl The logarithmic Sobolev and Sobolev inequalities along the Ricci flow}. Commun. Math. Stat. {\bf 3} (2015), no. 1, 1-36
%
%\bibitem{Yang1992a} Yang, D., {\sl Convergence of Riemannian manifolds with integral bounds on curvature}. I. Ann. Sci. \'Ecole Norm. Sup. (4) 25 (1992), no. 1, 77--105.
%
%\bibitem{Yang1992b} Yang, D., {\sl Convergence of Riemannian manifolds with integral bounds on curvature}. II. Ann. Sci. \'Ecole Norm. Sup. (4) 25 (1992), no. 2, 179--199.
%
%
%
\bibitem{Zhang2007}Zhang, Q., {\sl A uniform Sobolev inequality under Ricci flow}, Int. Math. Res. Not. (2007), rnm056.
%
%\bibitem{ZhangBook}Zhang, Q., {\sl Sobolev inequalities, heat kernels under Ricci flow, and the Poincar\'e conjecture}. CRC Press, Boca Raton, FL, 2011. x+422 pp. ISBN: 978-1-4398-3459-6
%%%%%%%%%%%%%%%%%%%%%%%%%%%%%%%%%%%%%%%%%%%%%%%%%%%%%%%%%%%%%%%%%%%%%%%%%%%%%%%%%%%%%%%%%%%%%%%%%%%%%%%%%%%%%%%%%%%%%%%%%%%%%%%%%%%%%%%%%%%%%%%%%%%%%%%%%%%%%%%%%%%%%%%%%%%%%%%%%%%%%%%%%%%%%%%%%%%%%%%%%%%%%%%%%%%%%%%%%%%%%%%%%%%%%%%%%%%%%%%%%%%%%%%%%%%%%%%%%%%%%%%%%%%%%%%%%%%%%%%%%%%%%%%%%%%%%%%%%


\end{thebibliography}
\end{document}